\documentclass[preprint,12pt]{amsart}
\usepackage{tikz}
\usetikzlibrary{matrix}

\usepackage{amssymb}
\usepackage{verbatim}
\usepackage{array}
\usepackage{latexsym}
\usepackage{enumerate}
\usepackage{amsmath}
\usepackage{amsfonts}
\usepackage{amsthm}
\usepackage{color}
\usepackage[english]{babel}
\usepackage[pagewise]{lineno}

\newtheorem{theorem}{Theorem}[section]
\newtheorem{proposition}[theorem]{Proposition}
\newtheorem{lemma}[theorem]{Lemma}
\newtheorem{corollary}[theorem]{Corollary}

{\theoremstyle{definition}
\newtheorem{definition}{Definition}}
\newtheorem{rem}[theorem]{Remark}

\def\S{\mathbf S}

\def\cX{\mathcal X}
\def\cY{\mathcal Y}

\def\Aut{\mbox{\rm Aut}}

\def\PG{{\rm{PG}}}

\def\Aut{\mbox{\rm Aut}}

\def\Ker{\mbox{\rm Ker}}

\newcommand{\PGL}{\mbox{\rm PGL}}

\newcommand{\aut}{\mbox{\rm Aut}}

\title{Large automorphism groups of ordinary curves in characteristic $2$}
\date{}

\author{Maria Montanucci}
\address{Dipartimento di Matematica, Informatica ed Economia\\ Universit\`a degli Studi  della Basilicata\\ Contrada Macchia Romana\\ 85100 Potenza (Italy)}
\email[MMontanucci]{maria.montanucci@unibas.it}

\author[PSpeziali]{Pietro Speziali}
\address[PSpeziali]{Instituto de Ci\^encias Matem\'aticas e de Computa\c{c}\~ao,\\ Universidade de S\~ao Paulo,\\S\~ao Carlos, SP 13560-970, (Brazil)}
\email[PSpeziali]{pietro.speziali@unibas.it}

\begin{document}
\maketitle

\begin{abstract}
Let $\cX$ be a (projective, non-singular,  irreducible) curve of even genus $g(\cX) \geq 2$ defined over an algebraically closed field $K$ of characteristic $p$. If the $p$-rank $\gamma(\cX)$ equals $g(\cX)$, then $\cX$ is \emph{ordinary}. In this paper, we deal with \emph{large} automorphism groups $G$ of ordinary curves. 
Under the hypotheses that  $p = 2$, $g(\cX)$ is even and $G$ is solvable, we prove that $|G| < 35(g(\cX) +1)^{3/2}$.
\end{abstract}



\section{Introduction}\label{intr}
Let $p >0$ be a prime. By a \emph{curve} $\cX$, we mean a  projective, non-singular, irreducible curve defined over an algebraically closed field $K$ of characteristic $p$. Usually, the study of the geometry of $\cX$ is carried out through the study of its invariants, such as its \emph{genus} $g(\cX)$, its $p$-\emph{rank} (or \emph{Hasse-Witt invariant}) $\gamma(\cX)$, and its \emph{automorphism group} $\aut(\cX)$. The genus is the dimension of the $K$-vector space of holomorphic differentials on $\cX$, while the $p$-rank is the dimension of the $K$-vector space  of holomorphic logarithmic differentials. A curve $\cX$ is \emph{ordinary} if $g(\cX) = \gamma(\cX)$. An automorphism  of $\cX$ arises by defining a field automorphism $\sigma$ of the function field $K(\cX)$ fixing the ground field $K$ elementwise.  From a purely geometric point of view, $\cX$ admits a non-singular model in some projective space $\PG(r,K)$ for some $r\leq g(\cX)$, and every automorphism of $\cX$ can be represented by a linear collineation in $\PGL(r+1,K)$ leaving $\cX$ invariant. 
 
 Henceforth, we shall denote the automorphism group of $\cX$ by $\aut(\cX)$.  By a classical result, $\aut(\cX)$ is finite whenever $g(\cX) \geq 2$.  Also, if the characteristic $p$ of $K$ divides $|\aut(\cX)|$, then the classical Hurwitz bound 
 $$
 |\aut(\cX)| \leq 84(g(\cX)-1)
 $$
 no longer holds in general. A major achievement here, due to Henn \cite{henn}, is the classification of all the curves having an automorphism group whose size exceeds $8g(\cX)^3$. 
On the one hand, as all such curves have zero $p$-rank, then an interesting problem is to find a (minimal) function $f(g(\cX))$ depending on the genus $g(\cX)$ such that all curves $\cX$ with $|\aut(\cX)| > f(g(\cX))$ have zero $p$-rank. Such a problem has gained much attention since the early 2000's, and 
one expects that such a function is $f(g(\cX)) = cg(\cX)^2$ for some constant $c$; see for instance \cite{giulietti-korchmaros-2017} for the case when $g(\cX)$ is even.

Loosely speaking, the hypothesis that $2 | g(\cX)$ gives strong restrictions on the structure of a Sylow $2$-subgroup of $\aut(\cX)$, whence the powerful tools from finite group theory can be exploited. On the other hand, a general curve is ordinary, so it is natural to seek for a bound on the size of the automorphism group $\aut(\cX)$ of ordinary curves. In the ordinary case, a bound of type $|\aut(\cX)| \leq c(p)g(\cX)^{8/5}$ for some constant $c(p)$ depending on $p$ is expected; see \cite{KR}. If $\aut(\cX)$ is solvable and $p >2$, an even tighter bound 

\begin{equation}\label{kmbound}
|\aut(\cX)| \leq 34(g(\cX)+1)^{3/2}
\end{equation}
holds; see \cite {montanucci- korchmaros}. The latter bound is sharp up to the constant term; see \cite{montanucci-korchmaros-speziali}. The key point here is that \emph{large} solvable subgroups of $\aut(\cX)$ have a prescribed structure; see \cite[Lemma 2.1]{montanucci- korchmaros} and \cite[Lemma 4.1]{giulietti-korchmaros-2017}.
 
The object of this paper is to establish the characteristic $2$ analogue of \eqref{kmbound}  when $g(\cX)$ is even, see Corollary \ref{mariaportamivia}. More in detail, under the aforementioned hypotheses, we prove that
$$
|\aut(\cX)| < 35(g+1)^{3/2}.
$$
 As a final remark, we point out that our result is not a straightforward generalization of the results in \cite{montanucci- korchmaros}, since in characteristic $2$ many problems arise which do not occur in the odd characteristic case.

\section{Background and Preliminary Results}\label{sec2}

Well-known references for the theory of curves and algebraic function fields are \cite{hirschfeld-korchmaros-torres2008} and \cite{stbook}. Let $\cX$ be a curve defined over an algebraically closed field $K$ of positive characteristic $p$ for some prime $p$. We denote by $K(\cX)$ the function field of $\cX$. By a point $P \in \cX$ we mean a point in a nonsingular model of $\cX$; in this way, we have a one-to-one correspondence between points of $\cX$ and places of $K(\cX)$. 

Let $\aut(\cX)$ denote the geometric automorphism group of $\cX$. For a subgroup $S$ of $\aut(\cX)$, we denote by $K(\cX)^S$ the fixed field of $S$. A nonsingular model $\bar{\cX}$ of  $K(\cX)^S$ is referred to as the quotient curve of $\cX$ by $S$ and denoted by $\cX/S$; also, the corresponding covering morphism is denoted by $\pi: \cX \rightarrow \cX/S$. 
The field extension $K(\cX):K(\cX)^S$ is Galois with Galois group $S$. 

For a point $P \in \cX$, $S(P)$ is the orbit of $P$ under the action of $S$ on $\cX$ seen as a point-set. The orbit $S(P)$ is said to be long if $|S(P)| = |S|$, short otherwise. There is a one-to-one correspondence between the points in all the short orbits and ramified points in the extension $K(\cX):K(\cX)^S$. It might happen that $S$ has no short orbits; if this is the case, the cover $\pi: \cX \rightarrow \cX/S$ (or equivalently, the extension $K(\cX):K(\cX)^S$) is unramified. When clear from the context, we will discuss the orbits of a subgroup of $\aut(\cX)$ without referencing the curve $\cX$.

For $P \in \cX$, the subgroup $S_P$ of $S$ consisting of all elements of $S$ fixing $P$ is called the stabilizer of $P$ in $S$.  For a non-negative integer $i$, the $i$-th ramification group of $\cX$ at $P$ is denoted by $S_P^{(i)}$, and defined by
$$
S_P^{(i)}=\{\sigma \ | \ v_P(\sigma(t)-t)\geq i+1, \sigma \in S_P\}, 
$$
 where $t$ is a local parameter at $P$ and $v_P$ is the respective discrete valuation. Here $S_P=S_P^{(0)}$. Furthermore, $S_P^{(1)}$ is the unique Sylow $p$-subgroup of $S_P^{(0)}$, and the factor group $S_P^{(0)}/S_P^{(1)}$ is cyclic of order prime to $p$; see \cite[Chapter IV]{serre}. In particular, if $S_P$ is a $p$-group, then $S_P=S_P^{(0)}=S_P^{(1)}$.

 For a point $ R \in  \cX/S$, we define  the ramification index $e_R$  and the different exponent $d_R$ as follows: take $P \in \cX$ which satisfies $\pi(P) = R$. Then, $e_R = |S^{(0)}_P|$ and $d_R = \sum_{i \geq 0}(|S_P^{(i)}|- 1)$. Note that, since $\pi$ is a Galois covering, $e_R$ and $d_R$ do not depend on the choice of $P$.

Let $g$ and $\bar{g}$ be the genus of $\cX$ and $\bar{\cX}=\cX/S$, respectively. The Hurwitz genus formula is 
\begin{equation}\label{rhg}
2g-2=|S|(2\bar{g}-2)+\sum_{P \in \cX}\sum_{i \geq 0}\big(|S_P^{(i)}|-1\big);
\end{equation}
see \cite[Theorem 11.72]{hirschfeld-korchmaros-torres2008}.
 If $\ell_1,\ldots,\ell_k$  are the sizes of the short orbits of $S$, then (\ref{rhg}) yields
\begin{equation}\label{rhso}
2g-2 \geq |S|(2\bar{g}-2)+\sum_{\nu=1}^{k} \big(|S|-\ell_\nu\big),
\end{equation}
and equality holds if $\gcd(|S_P|,p)=1$ for all $P \in \cX$; see \cite[Theorem 11.57 and Remark 11.61]{hirschfeld-korchmaros-torres2008}.

Let $\gamma = \gamma(\cX)$ denote the $p$-rank (equivalently, the Hasse-Witt invariant) of $\cX$. If $S$ is a $p$-subgroup of $\aut(\cX)$ then
the Deuring-Shafarevich formula, see \cite[Theorem 4.2]{subrao}, states that
\begin{equation}
    \label{eq2deuring}
\gamma-1={|S|}(\bar{\gamma}-1)+\sum_{i=1}^k (|S|-\ell_i),
    \end{equation}
where $\bar{\gamma} = \gamma(\cX/S)$ is the $p$-rank of $\cX/S$ and $\ell_1,\ldots,\ell_k$ denote the sizes of the short orbits of $S$. Both the Hurwitz and Deuring-Shafarevich formulas hold true for rational and elliptic curves provided that $G$ is a finite subgroup.

A subgroup of $\aut(\cX)$ is a \emph{$p'$-group} (or \emph{a prime to $p$} group) if its order is prime to $p$. 
A subgroup $G$ of $\aut(\cX)$ is \emph{tame} if the  stabilizer of any point in $G$ is a $p'$-group. Otherwise, $G$ is \emph{non-tame} (or \emph{wild}). 
If $G$ is tame, then the classical Hurwitz bound $|G|\leq 84(g(\cX)-1)$ holds, but for non-tame groups this is far from being true. 

In this paper we deal with \emph{ordinary curves}, that is, curves for which equality $g(\cX) = \gamma(\cX)$ holds. We now collect some results regarding automorphism groups of ordinary curves. 
\begin{theorem}[\cite{nakajima 1987}, Theorem 3] \label{thm84g2}
Let $\cX$ be an ordinary curve with $g(\cX) \geq 2$. Then the following inequality 
\begin{equation}
|\aut(\cX)| \leq 84(g(\cX)-1)g(\cX)
\end{equation}
holds.
\end{theorem}

\begin{theorem}[\cite{nakajima 1987}, Theorem 2(i)]\label{2i}
Let $\cX$ be ordinary, and let $G$ be a finite subgroup of $\aut(\cX)$. Then for every point $P$ of $\cX$,  $G_P^{(2)}$ is trivial. 
\end{theorem}

Theorem \ref{2i} yields a nice simplification of the Hurwitz formula. As we will use it a number of times, we state it once and for all in the following corollary.

\begin{corollary}\label{2ii}
Let $\cX$ be ordinary, and let $G$ be a finite subgroup of $\aut(\cX)$. Let $g$ and $\bar{g}$ be the genus of $\cX$ and $\bar{\cX}=\cX/G$, respectively; then the Hurwitz genus formula for $\cX \rightarrow \cX/G$ reads
\begin{equation}\label{rhord}
2g-2=|G|(2\bar{g}-2)+\sum_{P \in \cX}\big(|G_P^{(0)}|+ |G_P^{(1)}|-2\big).
\end{equation}
\end{corollary}

We end this section recalling some definitions from Group Theory and group actions. 

\begin{definition}
A subgroup $H$ of a group $G$ is said to be a \emph{minimal normal subgroup} if $H$ is normal in $G$ and any normal subgroup of $G$ properly contained in $H$ is trivial.
If $G$ is solvable then a minimal normal subgroup $H$ of $G$ is elementary abelian, see \cite[Theorem 5.46]{machi}.
\end{definition}

\begin{definition}
For a group $G$, the \emph{odd core} $O(G)$ is its maximal normal subgroup of odd order. A group is \emph{odd core-free} if $O(G)$ is trivial. 
\end{definition}

\begin{definition}
Let $G$ be a group acting on a  set $X$. G is said to act \emph{semiregularly} on $X$ if the stabilizer $G_x$ of any $x \in X$ is trivial.
\end{definition}

\begin{definition}
A group $G$ acts \emph{n-transitively} on a set $X$ if for any two $n$-tuples $(x_1,\ldots,x_n)$, $(y_1,\ldots,y_n)$ of distinct elements in $X$ there is some $g \in G$ with $x_i^g =y_i$ for $i=1,\ldots,n$.
The group $G$ acts \emph{sharply $n$-transitively} on $X$ if this $g \in G$ is unique.
\end{definition}

\section{Solvable subgroups of $\aut(\cX)$ for $p = 2$}\label{sec:p=2}

Throughout this Section, $\cX$ and $\cY$ are curves with $g(\cX), g(\cY) \geq 2$ defined over an algebraically closed field $K$ of characteristic $2$. Also,  $G \leq \aut(\cX)$ denotes a solvable automorphism group of $\aut(\cX)$. 
The aim of this section is to prove the bound 
\begin{equation} \label{bound}
|G| < 35(g(\cX)+1)^{3/2}
\end{equation}
when $\cX$ is ordinary of even genus. 

\begin{rem} \label{remtame}
 If $g(\cX) \geq 2$ then $84(g(\cX)-1) \leq 35(g(\cX)+1)^{3/2}$. Hence in the proof of Theorem \ref{prop1} and Theorem \ref{prop2} we will assume that the classical Hurwitz bound for $G$ is not satisfied. From \cite[Theorem 11.56]{hirschfeld-korchmaros-torres2008}, this implies that $|G|$ is even.
 \end{rem}
The following lemma describes the short orbits of large solvable automorphism groups of a curve with respect to its $p$-rank.

\begin{lemma}\label{terribile} 

Let $H$ be a solvable subgroup of $\aut(\cX)$ of even order containing a unique Sylow $2$-subgroup $Q$.
Suppose that a complement $U$ of $Q$ in $H$ is abelian.  If
\begin{equation}
\label{eq16nov2016}
{\mbox{$|H| > 12(g-1)$}},
\end{equation}
Then $\cX/Q$ is rational. Further, $U$ is cyclic, $H$ has two short orbits, and one of the following holds.
 \begin{itemize}
\item[\rm(i)]  $\cX$ has positive $p$-rank, $Q$ has exactly two (non-tame) short orbits, and they are also the only short orbits of $H$; 
\item[\rm(ii)] $\cX$ has zero $p$-rank and $H$ fixes a point.
\end{itemize}
\end{lemma}
\begin{proof}
By the Schur-Zassenhaus Theorem, we have $H=Q\rtimes U$, that is, $H$ is a semidirect product of the form $Q\rtimes U$. Set $|Q|=2^r$, $|U|=u$.

We prove that $\bar{\cX} = \cX/ Q$ is rational by assuming that $g(\bar{\cX}) = \bar{g} \geq 2$ or $\bar{g} = 1$ and finding a contradiction (see Cases $1$ and $2$ below).

\textbf{Case 1:} $\bar g\geq 2$. Since $\aut(\bar{\cX})$ has a subgroup isomorphic to $U$, \cite[Theorem 11.79]{hirschfeld-korchmaros-torres2008} yields $4\bar g+4\geq |U|.$ Furthermore, from
the Hurwitz genus formula applied to the covering $ \cX \rightarrow \cX/Q$, we have $g-1\geq |Q|(\bar g-1)$. Therefore, from  \eqref{eq16nov2016}
$$(4\bar g+4)|Q|\geq |U||Q|=|H|>12(g-1) \geq 12 |Q|(\bar g-1),$$
whence, since $\bar{g} \geq 2$, we get $$12<4 \ \frac{\bar g+1}{\bar g-1}\leq 12,$$
a contradiction.

\textbf{Case 2:} $\bar g =1$. In this case $\cX\rightarrow \bar{\cX}$ ramifies, otherwise $\cX$ itself would be elliptic. Thus, there exists a point $P \in \cX$ with non-trivial stabilizer $Q_P \leq Q$. Let $o$ be the orbit of $P$ under the action of $H$. Then, as $Q$ is normal in $H$, $o$ is the union of some  
 $Q$-short orbits, say $o = o_1 \cup o_2 \cup \ldots \cup o_{u_1}$, where $|o_i| = 2^v$ for some $v< r$ and every $i = 1,\ldots,u_1$. 
In particular, $o$ is an $H$-orbit of size $u_12^v$. 

 Equivalently, the stabilizer $H_P$ of $P$ has order
$2^{r-v}u/u_1$, and it is a semidirect product $Q_1\rtimes U_1$ where $|Q_1|=2^{r-v}$ and $|U_1|=u/u_1$ for a subgroup $Q_1$ of $Q$ and $U_1$ of $U$ respectively.
The point $\bar{P}$ lying under $P$ in the covering $\cX\rightarrow \bar{\cX}$ is fixed by the factor group $\bar{U_1}=U_1Q/Q$. Since $\bar{\cX}$ is elliptic, \cite[Theorem 11.94]{hirschfeld-korchmaros-torres2008} implies $|\bar{U}_1|\leq 3$ and so $u_1 \geq u/3$. 
From the Hurwitz genus formula applied to  $\cX \rightarrow \bar{\cX}$,
$$2g-2\geq 2^v u_1 \sum_{i \geq 0}(Q_P^{(i)}-1)\geq 2^v u_1 \cdot 2(2^{r-v}-1) \geq 2^v \frac{u}{3} \cdot 2(2^{r-v}-1) \geq 2^v \frac{u}{3} (2^{r-v})=\frac{2^vu}{3}=\frac{|H|}{3},$$
where the fist inequality comes from the fact that we are just looking at points in $o$, the second inequality follows from the fact that we are looking to the first two ramification groups, and the last inequality follows from $v< r$.
Thus $|H| \leq 6(g-1)$, contradicting (\ref{eq16nov2016}).

Hence, we get $\bar g=0$. In this case, $Q$ has at least one short orbit. Furthermore, $\bar{U}=UQ/Q$ is isomorphic to a subgroup of $\PGL(2,K)\cong\aut(\bar{\cX})$. 

Since $U\cong \bar{U}$ has odd order and $\rm{char}(K)=2$,  the classification of finite subgroups of $\PGL(2,K)$, see \cite[Theorem 1]{maddenevalentini1982}, shows that $U$ is a cyclic group, $\bar{U}$ fixes two points $\bar{P}_0$ and $\bar{P}_\infty$ but no non-trivial element in $\bar{U}$ fixes a point other than $\bar{P}_0$ and $\bar{P}_\infty$.

Let $o_\infty$ and $o_0$ be the $Q$-orbits lying over $\bar{P}_0$ and $\bar{P}_\infty$, respectively. 
Obviously, $o_\infty$ and $o_0$ are short orbits of $H$. We show that $Q$ has at most one  or two short orbits, the candidates being $o_\infty$ and $o_0$.

By way of contradiction, assume there is a $Q$-orbit $o$ of size $2^m$ with $m<r$ which lies over a point $\bar{P}\in \bar{\cX}$ different from both $\bar{P}_0$ and $\bar{P}_\infty$. Since the orbit of $\bar{P}$ in $\bar{U}$ has length $u$, then the $H$-orbit of a point $P\in o$ has length $u2^m$. 

If $u>3$, that is $u \geq 5$, the Hurwitz genus formula applied to $\cX \rightarrow \cX/Q$ gives
 $$2g-2\ge -2 \cdot2^{r}+u2^m\cdot 2(2^{r-m}-1)\geq -2 \cdot 2^k+u2^m(2^{r-m})=2^{r}(u-2)>2^{r}\frac{u}{2}=\frac{|H|}{2},$$
a contradiction to (\ref{eq16nov2016}).

If $u=3$ then 
 $$2g-2\ge -2\cdot 2^{r} +3\cdot2^m \cdot 2(2^{r-m}-1) \geq -2 \cdot 2^{r} +3\cdot2^m(2^{r-m})=2^{r}=|Q|,$$
hence $|H|=3|Q| \leq 6(g-1)$ contradicting (\ref{eq16nov2016}).

We proved that $H$ has exactly two short orbits $o_0$ and $o_\infty$. Assume that they are both short orbits of $Q$. If their lengths are $2^a$ and $2^b$ with $a,b<r$, the Deuring-Shafarevich formula applied to $\cX\rightarrow \cX/Q$ gives
$$\gamma(\cX)-1=-2^{r}+(2^{r}-2^a)+2^{r}-2^b,$$
whence $\gamma(\cX)=2^{r}-(2^a+2^b)+1 \geq 2^{r} -2 \cdot 2^{r-1}+1>0$. This leads to case (i).

For case (ii), the same argument shows that if $Q$ has just one short orbit, say $o_0$ of length $2^a$ with $a<r$, then $\gamma(\cX)=0$ and the short orbit consists of a single point $P$. In fact in this case the Deuring-Shafarevich formula applied to $\cX\rightarrow \cX/Q$ gives

$$\gamma(\cX)-1=-2^{r}+(2^{r}-2^a)=-2^a,$$
which implies $\gamma(\cX)=0$ and $a=0$ and so $o_0=\{P\}$. Since $Q$ is a normal subgroup of $H$, and the \'etale fundamental group of the projective line is trivial (see \cite[Exerxise 6.22]{lenstra}, $P$ is also fixed by $H$. This leads to case (ii). 
\end{proof}

From now on, we assume that $\cX$ is ordinary of even genus.
We first show that this gives strong restrictions both on the structure and on the action of a Sylow $2$-subgroup of $\aut(\cX)$.

\begin{lemma} \label{2syl}

Let $S$ be a non-trivial Sylow $2$-subgroup of $\aut(\cX)$.
  Then $S$ is an elementary abelian $2$-group fixing an odd number $n \geq 1$ of points of $\cX$.
\end{lemma}

\begin{proof}
Let $\cX^\prime$ be the quotient curve $\cX/S$ and $g^\prime=g(\cX^\prime)$. From Corollary \ref{2ii}, 
$$g-1=|S|(g^\prime -1)+ \sum_{P \in \Delta} (|{S}|-1)+\sum_{i=1}^{k} \ell_i ( |S| / \ell_i-1)$$
where $\Delta=\{P \in \cX \mid S_P=S \}$, and $\ell_1 ,\dots, \ell_k$ are the lenghths of the  short orbits of $S$ of length at least equal to $2$. 

Since $|S|(g^\prime -1)$ and $\sum_{i=1}^{k} \ell_i ( |S_P| / \ell_i-1)$ are both even or zero while $g-1$ is odd, $\sum_{P \in \cX} (|{S}_P|-1)$ cannot vanish. Thus, 
$\Delta$ has odd length $n$, which is in particular at least equal to $1$. Also, by \cite[Corollary of Theorem 2]{nakajima 1987}, $S$ is an elementary abelian $2$-group.

\end{proof}

To prove the bound \eqref{bound}, we apply an inductive argument on the genus $g(\cX)$. The main steps of our proof can be described as follows.

\begin{itemize} \item If $g(\cX)=2$ then $|G| \leq 48$ from \cite[Proposition 11.99]{hirschfeld-korchmaros-torres2008}, whence the bound $|G| \leq 35(g(\cX)+1)^{3/2}$  holds. 
\item Suppose that $g(\cX) > 2$. Two cases are distinguished.   \begin{enumerate} \item If the odd-core $O(G)$ of $G$ is non-trivial, then consider the quotient curve $\cX/O(G)$. We prove that for $g(\cX/O(G))=0$ the bound $|G| \leq 35(g(\cX)+1)^{3/2}$ is satisfied, while the case $g(\cX/O(G))=1$ is impossible. For $g(\cX/O(G)) \geq 2$, we prove that $\cX/O(G)$ is ordinary and that $g(\cX/O(G))$ must be even. Hence, the induction hypothesis holds for the quotient group $G/O(G)$ of automorphisms of the quotient curve $\cX/O(G)$ since $g(\cX/O(G))<g(\cX)$, so that $|G/O(G)| \leq 35(g(\cX/O(G))+1)^{3/2}$. If $|G| > 35(g(\cX)+1)^{3/2}$, then the Hurwitz genus formula applied to \\$\cX \rightarrow \cX/O(G)$ yields   $|G/O(G)| > 35(g(\cX/O(G))+1)^{3/2}$, a contradiction. 
\item If $O(G)$ is trivial then $G$ admits a minimal normal subgroup which is an elementary abelian $2$-group. Let $Q$ be the largest normal $2$-subgroup of $G$. We apply to $Q$ the same strategy as for the previous case, that is, we first analyze the cases $g(\cX/Q)=0$ and $g(\cX/Q)=1$, and then we apply the induction hypothesis to the quotient group $G/Q$ seen as an automorphism group of the quotient curve $\cX/Q$ when $g(\cX/Q) \geq 2$.  In this way, if $|G| > 35(g+1)^{3/2}$, then we obtain a contradiction.
 \end{enumerate} \end{itemize} Note that, for the inductive argument to work, it is necessary to ensure that the quotient curves $\cX/O(G)$ and $\cX/Q$ both have even genus if they are neither rational nor elliptic. The following lemmas will be used for this purpose, see Proposition \ref{prop20nov2016} and Theorem \ref{prop2}. 


\begin{lemma}\label{prop19nov2016}
 Let $\cY$ be a curve of genus $g(\cY) \geq 2$. Let $H$ be a solvable subgroup of even order of $\aut(\cY)$ satisfying the following conditions.
\begin{itemize}
\item[(I)] A Sylow $2$-subgroup of $H$ fixes a point of $\cY$.
\item[(II)] The action of $H$ on the set of points whose stabilizer contains a Sylow $2$-subgroup of $H$ has an odd number of orbits. 
\item[(III)] At any point of $\cY$, the second ramification group of $H$ is trivial.
 \end{itemize}
Then  $g(\cY)$ is even.
\end{lemma}
\begin{proof} Let $o_1,\ldots, o_k$ denote the non-tame short orbits of $H$. Choose one point $P_i$ from $o_i$, and denote by $S_{P_i}$ the Sylow $2$-subgroup of the stabilizer of $P_i$ in $H$. 

From (III), $S_{P_i}^{(2)}$ is trivial. If $H$ also has tame short orbits, say $\theta_1,\ldots,\theta_m$, choose one point $Q_j$ from each of them.
From the Hurwitz genus formula applied to $\cX \rightarrow \cX/H$,
\begin{equation}
\label{eq120nov2016}
2(g(\cY)-1)=2|H|(g(\cY/H)-1)+\sum_{i=1}^k |o_i|(|H_{P_i}|-1+|S_{P_i}|-1)+\sum_{j=1}^m |\theta_j|(|H_{Q_j}|-1).
\end{equation}
The stabilizer of a point in any tame short orbit is of odd order while the length of any tame short orbit is even. Therefore, $\sum_{j=1}^m |\theta_j|(|H_{Q_j}|-1)$ is divisible by $4$.
Furthermore, if $S_{P_i}$ for some $1\le i \le k$ is not a Sylow $2$-subgroup of $H$ then $|o_i|=|H|/|H_{P_i}|$ is even, and hence $|o_i|(|H_{P_i}|-1+|S_{P_i}|-1)$ is also divisible by $4$.
If $S_{P_i}$ for some $1\le i \le k$ is a Sylow $2$-subgroup of $H$ then $|o_i|=|H|/|H_{P_i}|$ is odd whereas
$$|H_{P_i}|-1+|S_{P_i}|-1=|S_{P_i}|\Big(\frac{|H_{P_i}|}{|S_{P_i}|}+1\Big)-2\equiv 2 \pmod 4.$$
Therefore, (II) yields that the positive integer in the right hand side of (\ref{eq120nov2016}) is congruent to $2 \pmod 4$ whence $g(\cY)-1$ is odd.
\end{proof}

By Remark \ref{remtame}, we can assume that $|G|>84(g(\cX)-1)$. From \cite[Theorem 11.56]{hirschfeld-korchmaros-torres2008}, $G$ has either $1$ or $2$ non-tame short orbits. In Lemma  \ref{prop19nov2016} we dealt with the former case, while in Lemma \ref{prop21nov2016} below we deal with the latter.

\begin{lemma}\label{prop21nov2016}
Let $\cY$ be such that $g(\cY) \geq 2$ and $\gamma(\cY) > 0$. Let $H$ be a solvable subgroup of $\aut(\cY)$ of order $|H| >84(g(\cY)-1)$ satisfying the following condition.
\begin{itemize}
\item[(IV)] $H$ has exactly two non-tame short orbits whose points are fixed by Sylow $2$-subgroups of $H$.
\end{itemize}
Then $H=S_2\rtimes U$ where $S_2$ is a (elementary abelian minimal normal) Sylow $2$-subgroup of $H$ with a cyclic complement $U$. Moreover, $\cY/{S_2}$ is rational and $g(\cY)$ is even.
\end{lemma}
\begin{proof} By contradiction. Take for $\cY$ a curve that is a minimal counterexample to Lemma \ref{prop21nov2016} with respect to the genus. From $|H|>84(g(\cY)-1)$ and (IV), $H$ has exactly two short orbits, say $o$ and $\theta$, both non-tame; see \cite[Theorem 11.56]{hirschfeld-korchmaros-torres2008}.

Now, let $P \in o$. Choose a minimal normal subgroup $T$ of $H$,  consider the quotient curve $\cY/T$  and the subgroup $\bar{H}=H/T$ of $\aut(\cY/T)$.
Recall that $T$ is elementary abelian as $H$ is solvable.  Furthermore, $\bar{H}$ has two non-tame short orbits  on $\cY/T$. In fact, the points lying under the points in $o$, as well as those lying under the points in $\theta$, form two orbits of $\bar{H}$, say $\bar{o}$ and $\bar{\theta}$, respectively. Also, let $S_2$ be a Sylow $2$-subgroup of $H$ fixing $P$, and  $\bar{S}_2=S_2T/T$. Finally, let $\bar{P}\in \bar{o}$. 


Three cases are separately investigated according to the value of $g(\cY/T)$.

\textbf{Case 1: $g(\cY/T)\geq 2$}. In this case the Hurwitz genus formula applied to $\cY \rightarrow \cY/T$ implies $g(\cY)-1\geq |T|(g(\cY/T)-1)$. This together with $|H|>84(g(\cY)-1)$ yields $|\bar{H}|>84(g(\cY/T)-1)$. 



We prove that $\bar{S}_2$ is a Sylow $2$-subgroup of $\bar{H}$. If $|T|$ is odd then $S_2\cong \bar{S}_2$ and the latter claim is trivial. If $|T|$ is even, then $T$ is a normal subgroup of $S_2$ as $T$ is an elementary abelian $2$-group and it is normal in $H$. Moreover, since $|\bar{S}_2|=|S_2|/|T|$  we have that $$\frac{|\bar{H}|}{|\bar{S}_2|}=\frac{|H|/|T|}{|S_2|/|T|}=\frac{|H|}{|S_2|}$$ is odd and hence $\bar{S}_2$ is a Sylow $2$-subgroup of $\bar{H}$. 

Note that $\bar{S}_2$ is non-trivial, otherwise $T$ would be a Sylow $2$-subgroup and hence $\bar{H}$ would be a tame subgroup whose order cannot exceed $84(g(\cY/T)-1)$. 

From  \cite[Theorem 11.56]{hirschfeld-korchmaros-torres2008}, $\bar{o}$ and $\bar{\theta}$ are the only short orbits of $\bar{H}$. Therefore, the pair $\cY/T, \bar{H}$ satisfies Condition (IV). Since $\cY$ is taken as a minimal counterexample to Lemma  \ref{prop21nov2016}, $\bar H= \bar S_2 \rtimes \bar U$, where $\bar U$ is tame and cyclic. This implies that $\bar H$ fixes the point $\bar P$ of $\cY/T$ and hence that the points in one of the two non-tame short orbits of $H$, say $o$, lie under the points of a $T$-orbit implying $|o| \leq |T|$.

If $|T|$ is odd then $H=S_2 \rtimes U$ where $|U|$ is odd. The quotient curve $\cY/{S_2}$ is either rational or elliptic because otherwise the tame quotient group $H / S_2$ would satisfy $|H / S_2| > 84(g(\cY/{S_2})-1)$, a contradiction. 

If $\cY/{S_2}$ is rational then $U$ is cyclic from \cite[Theorem 1]{maddenevalentini1982}, a contradiction. Therefore $\cY/{S_2}$ is elliptic. Since $|o| \leq |T|$ the quotient group $H/S_2$ contains a subgroup of order at least $|U|/|T|$ fixing a point $P^\prime$ of $\cY/{S_2}$. From \cite[Theorem 11.94 (ii)]{hirschfeld-korchmaros-torres2008} we have that either $U=T$ or $|U|=3|T|$.

In the former case, as $U$ is abelian, from Lemma \ref{terribile} $U$ is cyclic, a contradiction. 

To deal with the latter case, note that the subgroup $\tilde{H}=\langle S_2, T \rangle = S_2 \times T$ is a subgroup of $H$ of index $3$. From Lemma \ref{terribile} $\tilde{H}$ fixes a point which is also a fixed point of $H$. Since this implies that $U$ is cyclic from \cite[Theorem 11.49]{hirschfeld-korchmaros-torres2008} we have a contradiction.

 Assume that $|T|$ is even. Since $U$ is isomorphic to $\bar U$ which is cyclic, we have a contradiction.

\textbf{Case 2: $g(\cY/T)=1$}. Then the Hurwitz genus formula applied to $\cY \rightarrow \cY/T$ gives
$$2(g(\cY)-1)\geq |o|(|T_P|-1)+ |\theta|(|T_R|-1)\geq \frac{|T|(|o|+|\theta|)}{2}$$
where $P\in o$ and $R\in \theta$. 

 Then
$H_P/T=H_PT/T$ is isomorphic to a subgroup of $\aut(\cY/T)$ fixing $\bar{P}$. From \cite[Theorem 11.94]{hirschfeld-korchmaros-torres2008}, $|H_PT/T|\leq 24$. Therefore, $|o|=|H|/|H_P|\geq \textstyle\frac{1}{24|T|}|H|$. The same holds for $R$. Therefore,
$2(g(\cY)-1)\geq \textstyle\frac{1}{24}|H|$ but this contradicts our assumption. 

\textbf{Case 3: $g(\cY/T)=0$}. 
Assume that $S_2 \ne T$. Since $P$ is fixed by $S_2$, then $\bar{S}_2$ fixes $\bar{P}$.
Now take a point $R\in \theta$ fixed by a Sylow $2$-subgroup $S_2^*$ of $H$. Since $\theta$ is an $H$-orbit and $S_2$ is conjugate to $S_2^*$ in $H$, there exists a point $R_1\in \theta$ of $\cY$ fixed by $S_2$. Hence the point $\bar{R}_1$ of $\cY/T$ lying under $S$ in the covering $\cY\rightarrow \cY/T$ is also fixed by $\bar{S}_2$. Since $\bar{P}\neq \bar{S}$, $\bar{S}_2$ has two fixed points, but this is impossible since the order of $S_2$ is a power of the characteristic of $K$, see \cite[Theorem 11.14]{hirschfeld-korchmaros-torres2008}. 

Therefore, $S_2=T$ and hence from the Schur-Zassenhaus Theorem $H=S_2\rtimes U$ with a subgroup $U$ of odd order. Furthermore  $\bar{H}\cong U$ has odd order. From the classification of finite subgroups of $\PGL(2,K)$ for $\rm{char}(K)=2$, $\bar{H}$ is a cyclic group. This shows that $H=S_2\rtimes U$ with a cyclic subgroup of odd order, a contradiction.

Finally, from Lemma \ref{terribile} (i), we have that $\cY$ has even genus.
\end{proof}

We are now in a position to prove the main results of this section, see Theorem \ref{prop1} and Theorem \ref{prop2}. As anticipated, we analyze the case in which the odd core of $G$ is non-trivial and the case in which it is trivial separately.  First, some technical lemmas are needed. 

\begin{lemma} \label{exclaim1}
Let $G \leq \aut(\cX)$ be solvable with non-trivial odd core $O(G)$.  
  Let $\bar \cX=\cX/{O(G)}$, $\bar G=G / O(G)$ and $\bar g=g(\bar \cX)$. Then $\bar{g}$ is even. Moreover, if $\Gamma=\{P \in \cX \mid |O(G)_P|>1\}$, then $|\Gamma|$ is even.
\end{lemma}
\begin{proof}
From the Hurwitz genus formula applied to $\cX \rightarrow \cX/O(G)$,
$$g-1=|O(G)|(\bar g-1)+\frac{1}{2} \sum_{P \in \Gamma} (|O(G)_P|-1).$$
Let $D=\sum_{P \in \Gamma} (|O(G)_P|-1) \geq 0$ and $S$ be a Sylow $2$-subgroup of $G$. 
 
If $D \equiv 0 \mod 4$, then the claim follows as both $g-1$  and $|O(G)|$ are odd. Now, let $|\Gamma|$ be even. If $D \equiv 2 \mod 4$, then we have an odd number of points whose stabilizer in $O(G)$ has order congruent to $3$ modulo $4$. Then, by the normality of $O(G)$, there is at least a short orbit $\Omega$ of $O(G)$ fixed by $S$ as otherwise we get $D \equiv 0 \mod 4$. Hence, there exists a point $P \in \Omega$ such that $S_P=S$ as $|\Omega|$ is odd. Since both $O(G)_P$ and $S$ are normal in $G_P$, we have that $\langle S,O(G)_P \rangle=S \times O(G)_P \leq G_P$ is abelian and a contradiction is obtained from \cite[Lemma 11.75 (iv)]{hirschfeld-korchmaros-torres2008} and Theorem \ref{2i}. More in detail, let $\bar{G} = S \times O(G)_P
$. Then \cite[Lemma 11.75 (iv)]{hirschfeld-korchmaros-torres2008} would yield $\bar{G}_P^{(2)} = \bar{G}_P^{(1)}= S$, which is impossible by Theorem \ref{2i}.

Therefore, we assume that $|\Gamma|$ is odd. 

 From Lemma \ref{2syl}, $S$ is an elementary abelian $2$-group fixing at least a point on $\cX$. Since $O(G)$ is normal in $G$, $S$ acts on $\Gamma$ and hence there exists a point $P \in \Gamma$ such that $S_P=S$ as $|\Gamma|$ is odd. Since both $O(G)_P$ and $S$ are normal in $G_P$ we have that $\langle S,O(G)_P \rangle=S \times O(G)_P \leq G_P$ is abelian and a contradiction is obtained as above. 
\end{proof}

\begin{lemma}\label{exclaim2}
Let $G\leq\aut(\cX)$ be of even order such that $O(G)$ is non-trivial and let $P\in\cX$. If $P$ is fixed by a $2$-subgroup $S$ of $G$, then $O(G)_P$ is trivial. 
\end{lemma}
\begin{proof}
By contradiction, assume that $O(G)_P$ is non-trivial. Then by \cite[Lemma 11.44 (ii) (c)]{hirschfeld-korchmaros-torres2008}, $O(G)_P$ is cyclic,  say $O(G)_P = \langle \alpha \rangle$. Then $S$ and $\alpha$ commute since  $O(G)_P=O(G) \cap G_P$ is a normal subgroup of $G_P$ while $S$ in normal in $G_P$  from \cite[Lemma 11.44 (ii) (e)]{hirschfeld-korchmaros-torres2008}. From \cite[Lemma 11.75 (iv)]{hirschfeld-korchmaros-torres2008}, this is a contradiction to Theorem \ref{2i}, whence our claim is proved.
\end{proof}

\begin{theorem} \label{prop1}
If the odd core $O(G)$ of $G$ is non-trivial then $|G| < 35(g +1)^{3/2}$. 
\end{theorem}

\begin{proof}

From Remark \ref{remtame} we can assume that $|G| > 84(g-1)$ so that $|G|$ is even.

The result is proved by induction on $g$. 

If $g=2$ then from \cite[Proposition 11.99 (ii)]{hirschfeld-korchmaros-torres2008}, $|G| \leq 48<35(g+1)^{3/2}$ and the claim follows. 

By induction hypothesis we assume that the claim holds for every ordinary curve $\cX^{\prime}$ of even genus $2 \leq g^{\prime} < g$, so that if $G^\prime$ is a solvable automorphism group of $\cX^\prime$ with $O(G^\prime) \ne \{1\}$, then $|G^\prime| \leq 35(g^\prime+1)^{3/2}$.

 Let $\bar \cX=\cX/{O(G)}$, $\bar G=G / O(G)$ and $\bar g=g(\bar \cX)$. 
From Lemma \ref{exclaim1} we conclude that either $\bar g \geq 2$ is even, or $\bar g=0$. We will analyze these two cases separately.

\textbf{Case 1:} $\bar{g} = 0$. In this case the covering $\cX \rightarrow \bar \cX$ ramifies. 
As $\bar G$ is isomorphic to a solvable subgroup of $\PGL(2,K)$ of even order, from \cite[Theorem 1]{maddenevalentini1982} $\bar G=E_{q} \rtimes C$, where $E_{q}$ is an elementary abelian $2$-group of order $q=2^h$ and $C$ is a cyclic group of odd order $c$.  Moreover, $\bar G$ fixes a point $\bar{P}_0$ of $\bar \cX $ and $E_{q}$ acts semiregularly on $ \cX \setminus \{\bar{P}_0\}$, while $C$ fixes also another point $\bar{P}_\infty$ acting semiregularly on $\bar{\cX } \setminus \{\bar{P}_0, \bar{P}_\infty\}$.

This implies that the $\bar G$-orbit of $\bar{P}_\infty$ has length $q$. Since $O(G)$ has odd order, the group $E_{q}$ is the image under reduction of a Sylow $2$-subgroup $S$ of $\aut(\cX)$, that is, $SO(G)/O(G) \cong E_{q}$. 

Denote by $O$ the $O(G)$-orbit lying over $\bar{P}_0$ in $\cX$. 

Using Lemma \ref{exclaim2} we are able to prove that $O$ is a long orbit for $O(G)$. 
In fact, as $S$ acts on $O$ and $|O|$ is odd, there exists a point $P \in O$ such that $S_P=S$. If $O$ is a short orbit of $O(G)$, then $O(G)_P$ is non-trivial and a contradiction to Lemma \ref{exclaim2} is obtained.

Thus, for $P \in O$ we have $G_P=S \rtimes U$, where $|U|=c$. Denote by $O_1,\dots,O_q$ the $O(G)$-orbits lying over the $q$ distinct points in $\bar \cX$ of the $\bar G$-orbit of $\bar{P}_\infty$, and let $\ell=|O_1|=\dots=|O_q|$. 

Assume that there exists a short orbit $\Sigma$ of $O(G)$ such that $\Sigma \ne O_i$ for any $i=1,\dots,q$ and let $T \in \Sigma$. In particular, in terms of the ramification of the cover $\cX\rightarrow \cX/O(G)$, $T$ does not lie over $\bar{P}_{\infty}$. From the Hurwitz genus formula  
$$2g-2\geq -2|O(G)|+qc \frac{|O(G)|}{|O(G)_T|}(|O(G)_T|-1),$$
as $\bar G$ acts on $\bar\cX \setminus \{\bar{P}_0, \bar{P}_\infty\}$. 

Hence,
$$2g-2 \geq |O(G)| (-2 +qc/2) \geq \frac{|O(G)|qc}{4}=\frac{|G|}{4},$$
a contradiction to $|G|>84(g-1)$.

Since $\cX \rightarrow \bar \cX$ ramifies we can suppose that $O_1,\dots,O_q$ are short orbits for $O(G)$. 

Write $|G|=|O(G)|qc$ and let $\mathcal{O}=O_1 \cup,\dots,\cup  \:O_q$. Clearly, $\mathcal{O}$ is a short orbit of $G$. Moreover, for every point $R \in O_1$,
$$|\mathcal{O}|=q|O_1|=\frac{q|O(G)|}{|O(G)_R|}, \quad {\rm and} \quad |G|=|G_R||\mathcal{O}|,$$
and hence $G_R$ is a cyclic tame group with $|G_R|=|O(G)_R|c$. From the Hurwitz genus formula
$$2g-2=-2|O(G)|+q|O_1|(|O(G)_R|-1) \geq -2|O(G)|+\frac{q|O_1||O(G)_R|}{2}$$
$$=|O(G)|(q/2-2)=|O(G)_R||O_1|(q/2-2) \geq \frac{|O(G)_R|q|O_1|}{5}=\frac{|G|}{5c}.$$
Thus, 
\begin{equation}
|G| \leq 10c(g-1)
\end{equation}
and since $|G|=|O(G)|qc$,
\begin{equation} \label{res1}
|O(G)|q \leq 10(g-1).
\end{equation}

From \cite[Proposition 1]{nakajima 1987}, $c \leq q-1$ and hence $c|O(G)| \leq 10(g-1)$ holds. 
The fact that $c \leq q-1$ can be seen also recalling that $C$ acts semi-regularly on $\Sigma$ where $\Sigma$ denotes the $\bar G$-orbit of $\bar P_0$ in $\cX/O(G)$ and $|\Sigma|=q$.  

If $qc \leq 12(g-1)$, then $c \leq 2\sqrt{3}\sqrt{g-1} +1$  as $c \leq q-1$. Combining with Equation (\ref{res1})
$$|G| \leq (2\sqrt{3}\sqrt{g-1}+1) \cdot  10(g-1)<35(g+1)^{3/2},$$
completing the proof of the Theorem.

Thus we can assume that $qc>12(g-1)$.  From Lemma \ref{terribile} (i), $S$ has just another short orbit $\Omega$ of length $q/\ell=2^r>1$ with $r<h$ other than the fixed point $P$; moreover $S$ and $G$ have the same short orbits $\{P\}$ and $\Omega$.

From Equation (\ref{res1})
$$c|O(G)|<q|O(G)|\le 10(g-1)<12(g-1)<qc,$$
and $|O(G)|<q$. 

Since this implies that $O$ cannot contain a long orbit of $S$
$$|O(G)|=1+|\Omega|=1+q / \ell=1+2^r,$$
and $G / \Ker\varphi$ acts sharply $2$-transitively 
 on $O$, where $\Ker\varphi$ denotes the Kernel of the permutation representation of $G$ on $O$. 

From \cite{Huppert} and \cite[Theorem 5.4.4]{bw}, as $G/ \Ker\varphi$ is a solvable $2$-transitive group, $1+2^r=|O|=p^t$ for some prime $p$ and $t \geq 1$ and so either $p=3$, $t=2$ and $k=3$ or $t=1$. 

If $|O(G)|=9$, then by  \cite[Theorem 20.7.1]{Marsh} the Sylow $2$-subgroup of $G / \Ker\varphi$ has order $8$ and must be isomorphic either to the quaternion group $Q_8$ or to a cyclic group, a contradiction to Lemma \ref{2syl}. 

Assume that $|O(G)|=p$. From the Deuring-Shafarevich formula applied to $S$, $g-1=\gamma-1=q-p$, and so $q= g+|O(G)|-1$. Combining with Equation (\ref{res1}),
$$(g-1)|O(G)|<(g+|O(G)|-1)|O(G)|=q|O(G)| \leq 10(g-1).$$

Since this implies that $1+q/ \ell=1+2^k=p=|O(G)|<8$ we have either $p=3$ or $p=5$. 

If $p=3$, $U$ and a subgroup $S^\prime$ of $S$ of index $2$ fix $O$ pointwise. Let $\cX^\prime=\cX/{S^\prime}$ and $g^\prime=g(\cX^\prime)$. From the Hurwitz genus formula applied to $\cX \rightarrow \cX^\prime$
$$2(q-3)=2(g-1)=2\frac{q}{2}(g^\prime-1)+2\cdot 3\bigg(\frac{q}{2}-1\bigg),$$
and hence $g^\prime=0$. The quotient group $S^\prime U / S^\prime \cong U$ is isomorphic to a subgroup of $\PGL(2,K)$ fixing three points, a contradiction to \cite[Theorem 1]{maddenevalentini1982}. 

Let $p=5$. If $q \geq 16$ then a Sylow $2$-subgroup of $G / \Ker\varphi$ has order at least $4$ and it must be cyclic being isomorphic to the multiplicative group of a field $\mathbb{F}$, a contradiction  to Lemma \ref{2syl}. The cases $q=4$ or $q=2$ are impossible as $\Omega$ is a short orbit of $S$. Assume that $q=8$. In this case either $c=1$ or $c=7$ from \cite[Proposition 1]{nakajima 1987} and $g=4$. If $c=1$ then $|G|=40<35(g+1)^{3/2}$, a contradiction. If $c=7$ then 
$|G|=280$.  Moreover, $U$ and a subgroup $S^\prime$ of $S$ of order $2$ fixes $O$ pointwise. Let $\cX^\prime=\cX/{S^\prime}$ and $g^\prime=g(\cX^\prime)$. From the Hurwitz genus formula applied to $\cX \rightarrow \cX^\prime$
$$2(q-p)=6=2(g-1)=2\cdot 2(g^\prime-1)+2\cdot 5 \bigg(2-1 \bigg)=4(g^\prime-1)+10,$$
and hence $g^\prime=0$. The quotient group $S^\prime U / S^\prime \cong U$ is isomorphic to a non-trivial subgroup of $\PGL(2,K)$ fixing $5$ points, a contradiction to \cite[Theorem 1]{maddenevalentini1982}. 

\textbf{Case 2:} $\bar g \geq 2$. By Lemma \ref{exclaim1} we have that $\bar g<g$ is even. The quotient group $\bar G$ is a solvable subgroup of $\aut(\bar \cX)$. We now prove that $\cX/O(G)$ is ordinary so that the induction hypothesis holds for $\bar G$ and hence $|\bar G| \leq 35(\bar g+1)^{3/2}$. To do that we first prove the following claims.

\textbf{Claim 1:} $\bar G_{\bar Q}^{(2)}$ is trivial for any $\bar Q\in \cX/{O(G)}$.

\begin{proof}[Proof of Claim 1]
Let $\bar Q\in {\cX/{O(G)}}$ and $\bar T$ be the Sylow $2$-subgroup of the stabilizer $\bar G_{\bar Q}=(G / O(G))_{\bar Q}$. Since $|O(G)|$ is odd, the automorphism group $\bar T$ is induced by some $2$-subgroup $T$ of $\aut(\cX)$ and $T \cong \bar T$. Since $\bar{T}$ is the Sylow 2-subgroup of $\bar{G}_Q$, the group $T$ acts on the $O(G)$-orbit $\Sigma$ which lies over the point $\bar{Q}$. Hence, recalling that $|\Sigma|$ is a divisor of $|O(G)| $ which is odd, $T$ fixes at least a point $Q \in \Sigma$.  
From Lemma \ref{exclaim2}, $\Sigma$ is a long orbit of $O(G)$. Therefore each local parameter $t \in K(\cX/{O(G)})$ at $\bar Q$, is also a local parameter at $Q$ in $K(\cX)$. 
Let $\bar\alpha\in\bar T$  and $\alpha\in T$ which induces $\bar\alpha$. Then
$$ v_Q(\alpha(t)-t) = v_Q(\bar\alpha(t)-t) = v_{\bar Q}(\bar\alpha(t)-t).  $$
Hence, $G_Q^{(i)} = \bar{G}_{\bar Q}^{(i)}$ for any $i\geq1$. In particular, $\bar G_{\bar Q}^{(2)}$ is trivial.
\end{proof} 

\textbf{Claim 2:} $\cX/O(G)$ is ordinary.

\begin{proof}[Proof of Claim 2]
Recall that $N_1=O(G)$ is the largest normal subgroup of $G$ whose order is prime to $2$. Also, by our hypothesis, the quotient curve $\cX_1=\cX/N_1$ is neither rational, nor elliptic, $g(\cX_1)=\bar g \geq 2$ is even and  $G_1=G/N_1$ is a solvable automorphism group of $\cX_1$ whose order satisfies $|G_1| > 84(g(\cX_1)-1)$.

Hence, $G_1$ has a minimal normal $d$-subgroup where $d$ must be equal to $2$ by the choice of $N_1$ to be the largest normal, prime to $2$ subgroup of $G$. 

Take the largest normal $2$-subgroup $N_2$ of $G_1$. We note that $N_2\neq G_1$. 

In fact, if $N_2=G_1$ then $G_1$ is a $2$-group of order bigger than $84(\bar g+1)>4 (\bar g-1)$. From \cite[Theorem 1]{nakajima 1987}, $\cX_1$ has zero $2$-rank, and hence
$G_1$ fixes a point $P_1\in \cX_1$, see \cite[Lemma 11.129]{hirschfeld-korchmaros-torres2008}. On the other hand, since ${G_1}_{P_1}^{(2)}$ is trivial from Claim 2,  \cite[Result 2.6]{montanucci- korchmaros} shows $|G_1|\le 2(\bar g-1)$, a contradiction.

Now, define $\cX_2$ to be the quotient curve $\cX_1/N_2$. Since the second ramification group of $N_1$ at any point of $\cX_1$ is trivial, \cite[Result 2.6 (i)]{montanucci- korchmaros} gives  $\bar g-\gamma(\cX_1)=|N_2|(g(\cX_2)-\gamma(\cX_2))$. In particular, if $\cX_2$ is ordinary or rational then $\cX_1$ is an ordinary curve.

We prove that the case $g(\cX_2)=1$ cannot occur. 

In fact, suppose that $g(\cX_2)=1$. In this case, $\cX_1\rightarrow\cX_1/N_2$ ramifies otherwise $\cX_1$ itself would be elliptic. Let $\Delta_1$ be a short orbit of $N_2$ and let $\Gamma$ be the short orbit of $G_1$ containing $\Delta_1$. Thus $\Gamma= \Delta_1 \cup \ldots \cup \Delta_k$ where $\Delta_i$ is a short orbit of $N_1$ with $|\Delta_i|=|\Delta_1|$ for every $i=2,\dots,k$.

 For a point $P \in \Delta_1$, $|{G_1}_P|=|G|/k|\Delta_1|$ and the quotient group ${G_1}_P N_2/N_2$ fixes the point $\bar P$ of $\cX_1/N_2$ which lies under the orbit $\Delta_1$. From \cite[Theorem 11.94]{hirschfeld-korchmaros-torres2008},
$$|{G_1}_P N_2 / N_2|=|{G_1}_P|/|{N_1}_P| \leq 24.$$
Moreover, from the Hurwitz genus formula applied to $\cX_1 \rightarrow \cX_1/N_2$,
$$2\bar g-2 \geq 2k|\Delta_1|(|{N_2}_P|-1) \geq \frac{2k|\Delta_1| |{N_2}_P|}{2} \geq \frac{k|\Delta_1| |{G_1}_P|}{24}=\frac{|{G_1}|}{24}.$$
 Since this implies that $|G_1| \leq 84(\bar g-1)$, we have a contradiction.

Therefore, $g(\cX_2)\geq 2$ with $g(\cX_2)>\gamma(\cX_2)$ may be assumed. The factor group $G_2=G_1/N_2$ is an automorphism group of the quotient curve $\cX_2=\cX_1/N_2$, and it has a minimal normal $d$-subgroup with $d\neq 2$, by the choice of $N_2$. 

Define $N_3$ to be the largest normal, prime-to-$2$ subgroup of $G_2$. Observe that $N_3$ must be a proper subgroup of $G_2$, otherwise $G_2$ itself would be a prime-to-$2$ subgroup of $\aut(\cX_2)$ of order bigger than $84(g(\cX_2)-1)$. Therefore, there exists a (maximal) non-trivial normal $2$-subgroup $N_4$ in the factor group $G_3=G_2/N_3$.

Now, the above argument remains valid whenever $G,N_1,G_1,N_2,\cX_1,\cX_2$ are replaced by $G_2,N_3,G_3,N_4,\cX_3,\cX_4$ where the quotient curves are $\cX_3=\cX_2/(G_2/N_3)$ and
$\cX_4= \cX_3/(G_3/N_4)$. In particular, we can suppose that $g(\cX_4)\neq 1$ and $g(\cX_3)-\gamma(\cX_3)=|N_4|(g(\cX_4)-\gamma(\cX_4))$.  

Repeating the above argument, a finite sharply decreasing sequence $g(\cX_1)>g(\cX_2)>g(\cX_3)>g(\cX_4) >\ldots$ arises. If this sequence has $n+1$ members then $g(\cX_n)-\gamma(\cX_n)=|N_{n+1}|(g(\cX_{n+1})-\gamma(\cX_{n+1}))$ with $g(\cX_{n+1})=\gamma(\cX_{n+1})=0$. Therefore, for some index $m\leq n$, the curve $\cX_m$ would not be ordinary, but the successive member $\cX_{m+1}$ would be an ordinary curve. Since $\cX_{m+1}$ is a quotient curve of $\cX_m$ with respect to a $2$-subgroup, this is impossible by \cite[Result 2.5 (v)]{montanucci- korchmaros}.
\end{proof}

From Claim 2, since $\cX/O(G)$ is an ordinary curve of even genus $\bar g<g$ with $\bar g \geq 2$ and $\bar G$ is a solvable group of automorphisms for $\cX/O(G)$, we have that $|\bar G| \leq 35(\bar g+1)^{3/2}$. 
If $|G| > 35(g+1)^{3/5}$, then the Hurwitz genus formula applied to $\cX \rightarrow \cX/O(G)$ implies that $2g-2 \geq |O(G)|(2\bar g-2)$. Thus, $|\bar G| \geq 35(\bar g + 1)^{3/2}$, a contradiction. Now, Theorem \ref{prop1} follows.

\end{proof}

We now deal with the case where $O(G)$ is trivial. Since this implies that $\Aut(\cX)$ has a minimal normal $2$-subgroup, we can consider the largest normal $2$-subgroup $Q$ of $\Aut(\cX)$, which is non-trivial. Proposition \ref{prop20nov2016} below will be used in the proof of Theorem \ref{prop2} to ensure that $g(\cX/Q)$ is even when $\cX/Q$ is neither rational nor elliptic.
\begin{proposition}\label{prop20nov2016}
Let $G$  be a subgroup of $\aut(\cX)$ satisfying the following conditions.
\begin{itemize}
\item[(i)] $G$ has a non-trivial normal $2$-subgroup $Q$, and $Q$ is not a Sylow $2$-subgroup of $G$.
\item[(ii)] $\cX/Q$ is neither rational nor elliptic.
\item[(iii)]  $|G|> 84(g(\cX)-1)$.
\end{itemize}
Then  $\bar g=g(\cX/Q)$ is even.
\end{proposition}
\begin{proof} We show that the subgroup $\bar{G}=G/Q$ of $\aut(\cX/Q)$ satisfies the conditions (I), (II) and (III) of Lemma \ref{prop19nov2016}.

From Lemma \ref{2syl}, we have that a Sylow $2$-subgroup $S_2$ of $G$ has a fixed point. Since $Q$ is a normal subgroup of $G$, $S_2$ contains $Q$. Note that the factor group $S_2/Q$ is non-trivial as $Q$ is properly contained in $S_2$ by (i). Also, $S_2/Q$ is a Sylow 2-subgroup of $\bar{G}$ and, seen as a subgroup of $\aut(\cX/Q)$ it fixes a points on $\cX$. 
This shows that Condition (I) is satisfied. 

Since $\cX$ is ordinary, so is $\cX/Q$, as $Q$ is a $2$-subgroup of $G$. 
In fact, let $\ell_1,\ldots,\ell_k$ be the short orbits of $Q$ on $\cX$. Fix a point $P_i \in \ell_i$ for $i=1,\ldots,k$. 
Then  Corollary \ref{2ii} yields $$2(g(\cX)-1)=2|Q|(g(\cX/Q)-1)+2\sum_{i=1}^{k}|\ell_i|\bigg(\frac{|Q|}{|\ell_i|}-1\bigg),$$
 while the Deuring-Shafarevic formula applied to $\cX\rightarrow \cX/Q$ gives
$$g(\cX)-1=\gamma(\cX)-1=|Q|(\gamma(\cX/Q)-1)+\sum_{i=1}^{k}|\ell_i|\bigg(\frac{|Q|}{|\ell_i|}-1\bigg),$$ implying that $g(\cX/Q)=\gamma(\cX/Q)$.

Therefore, Condition (III) is satisfied.

It remains to investigate Condition (II).
By (ii) and (iii), $G$ has  either one or two non-tame short orbits and in the latter case $G$ has no tame short orbits; see \cite[Theorem 11.56]{hirschfeld-korchmaros-torres2008}. 

Let $o$ be a non-tame short orbit 
whose points are fixed by Sylow $2$-subgroups of $G$. Choose one Sylow $2$-subgroup, say $S_2$, and consider the set $\rho$ of all fixed points of $S_2$ in $o$. 

The normalizer $N=N_G(S_2)$ leaves $\rho$ invariant. Actually, $N$ is transitive on $\rho$. In fact, if $P,R\in \rho$ are fixed by $S_2$, then there exists $g\in G$ that takes $P$ to $R$. Then the Sylow $2$-subgroup $gS_2g^{-1}$ of $G$ fixes $R$. Therefore, $R$ is fixed by both $S_2$ and $gS_2g^{-1}$. Hence $S_2=gS_2g^{-1}$ as $G_R$ has a unique maximal $2$-subgroup. Thus $g\in N$. This shows that $|\rho|=|N|/|N_P|$ is odd. 

By Lemma \ref{2syl}, (i) and (ii) yield that $S_2$ has an odd number of fixed points. This rules out  the possibility that $G$ has two non-tame short  orbits both consisting of fixed points of Sylow $2$-subgroups of $G$. In fact, since $|\rho|$ is odd, if both the non-tame short orbits contained fixed points from the same Sylow $2$-subgroup then $S_2$ would have an even number of fixed points. Otherwise, there would exist two Sylow $2$-subgroups of $G$ whose fixed points are in different (non-tame short) orbits of $G$, respectively. But this is not possible, as all the Sylow $2$-subgroups of $G$ are pairwise conjugate.

 Therefore, just one $G$-orbit, say $o$, comprises the fixed points of the  Sylow $2$-subgroups of $G$. 

Assume that $G$ has another non-tame short orbit, say $\theta$. Since $Q$ is a normal subgroup of $G$, $\theta$ is partitioned in $Q$-orbits with the same length $t$. 
Obviously, $t>1$ if and only if $Q$ does not fix any point in $\theta$. We show that this actually the case.

By contradiction, $t=1$. Take a point $P$ from $\theta$. Since $Q$ is contained in all Sylow $2$-subgroups, the Hurwitz genus formula applied to $\cX \rightarrow \cX/Q$ reads
\begin{equation}
\label{eq220nov2016}
2(g(\cX)-1)=2|Q|(g(\cX/Q)-1)+2|o|(|Q|-1)+2|\theta|(|Q|-1).
\end{equation}
whence $\theta$ is even.

 Let $\Delta \ne \{P\}$ be a $Q$-orbit contained in $\theta$ and let $\bar\theta$ be the set of the points of $\cX/Q$ lying under those in $\theta$. Now, we are in a position to prove that the factor group $\bar{G}=G/Q$ viewed as a subgroup of $\aut(\cX/Q)$ satisfies Condition (II). The Sylow $2$-subgroups of $\bar{G}$ are the factor groups $S_2/Q$ where $S_2$ ranges over the Sylow $2$-subgroups of $G$. Furthermore, the covering $\cX\rightarrow \cX/Q$ is totally ramified at the points lying under the points in $o$. Those points of $\cX/Q$ form a set $\bar{o}$ with $|\bar{o}|=|o|$, and $\bar{o}$ is a set of fixed points of Sylow $2$-subgroups of $\bar{G}$. 

Now either $S_2$ preserves $\Delta$ or not. In the former case, $\bar o$ and $\bar \theta$ are exactly the non-tame short orbits of the quotient group $\bar G$ and they are both consisting of fixed points of Sylow $2$-subgroups of $\bar G$. Thus, the claim follows from Lemma \ref{prop21nov2016}. In the latter case, $S_2$ has no fixed points on $\theta$ and thus $\bar G$ satisfies Condition (II). Now the claim follows from Lemma \ref{prop19nov2016}.
\end{proof}

\begin{theorem} \label{prop2}

If the odd core $O(G)$ of $G$ is trivial then $|G| < 35(g+1)^{3/2}$. 
\end{theorem}

\begin{proof} 
From Remark \ref{remtame} we can suppose that $|G|>84(g-1)$. 
The result is proved by induction on $g$. 

If $g=2$ then from \cite[Proposition 11.99 (ii)]{hirschfeld-korchmaros-torres2008}, $|G| \leq 48<35(g+1)^{3/2}$ and the claim follows. 

By the induction hypothesis we assume that the claim holds for every ordinary curve $\cX^\prime$ of even genus $2 \leq g^\prime<g$, so that if $G^\prime$ is an odd core-free solvable automorphism group of $\cX^\prime$ then $|G^\prime| < 35(g^\prime+1)^{3/2}$. 

Since $G$ is solvable and $O(G)$ is trivial, then $G$ admits a normal $2$-subgroup. Let $Q$ be the largest normal $2$-subgroup of $G$, $\bar g= g(\cX/Q)$ and $\bar G=G/Q$. 

Three cases are distinguished according as $\cX/Q$ is rational, elliptic, or $\bar g \geq 2$.

\textbf{Case 1:} $\bar g=0$. The quotient group $\bar G$ has no non trivial normal $2$-subgroup and it is isomorphic to a subgroup of $\PGL(2,K)$ with $\rm{char}(K)=2$. From \cite[Theorem 1]{maddenevalentini1982}, either $\bar G$ is a tame cyclic group or $\bar G$ is isomorphic to the alternating group $\rm{Alt}_4$ or $\bar G$ is isomorphic to the symmetric group $\rm{Sym}_4$. 

In the last two cases from \cite[Theorem 1(i)]{nakajima 1987},
$$|G| \leq 24|Q| \leq 96(g-1).$$
Since $g \geq 2$ we have that $96(g-1) < 35(g+1)^{3/2}$ and the claim follows. Therefore $\bar G$ is a tame cyclic group and hence $Q \in \rm{Syl}_2(G)$. 

From the Schur-Zassenhaus Theorem $G=Q \rtimes U$ where $U \cong \bar G$ is tame and cyclic. 

Since $|G|>12(g-1)$, from Lemma \ref{terribile} $Q$ has a fixed point $P$ on $\cX$ and exactly another short orbit $\Omega$ which is non-trivial. Moreover $G=G_P$ and $g=|Q|-|\Omega|$. 

Let $V$ be the subgroup of $Q$ fixing a point of $\Omega$. Then $V$ fixes $\Omega$ pointwise and since $U$ normalizes $Q$, then $U$ normalizes $V$. Therefore the representation of the action of $U$ on $Q$ by conjugation is completely reducible from Maschke's Theorem; see \cite[Theorem 6.1]{machi}. This means that $Q=V \times M$, where $M$ acts sharply transitively on $\Omega$. 

Let $g^\prime$ be the genus of the quotient curve $\cX/M$. 
Since the quotient group $G/M$ has even order and fixes two points in $\cX/M$, then $g^\prime \ne 0$ from \cite[Theorem 1]{maddenevalentini1982}. 

Assume that $g^\prime=1$. From the Hurwitz genus formula
$$2(|Q|-|M|-1)=2(g-1)=2(|M|-1),$$
and hence
$$|Q|=2|M|.$$
Furthermore, $|G/M|=2|U|$ and $G/M$ fixes two points in $\cX/M$. From \cite[Theorem 11.94]{hirschfeld-korchmaros-torres2008}, $|U|=3$ and hence from \cite[Theorem 1]{nakajima 1987}
$$|G|=|Q||U| \leq 12(g-1).$$
Thus, this case cannot occur as $|G|>84(g-1)$.

Assume that $g^\prime \geq 2$. From the Hurwitz genus formula
$$2(|Q|-|M|-1)=2(g-1)=2|M|(g^\prime-1)+2(|M|-1),$$
and thus
$$g^\prime=|Q|/|M| -1=|V|-1.$$
The quotient group $G/M \cong V \rtimes U$ has two fixed points of $\cX/M$, and $(\cX/M)/V \cong \cX/Q$ is rational. From the proof of Theorem 3 in \cite{nakajima 1987} (see Equation (4.16) and page 606), $|U| \leq \sqrt{|V|}+1=\sqrt{g^\prime+1}+1=\sqrt{g/|M|+1}+1$. Thus, 
$$|G|=|Q||U| \leq (g+|M|) (\sqrt{g/|M|+1}+1) \leq (g+|M|)(2\sqrt{2}\sqrt{g/ |M|}).$$
From \cite[Theorem 11.78 (i)]{hirschfeld-korchmaros-torres2008} $|M| \leq g$ and
$$|G| \leq 2\sqrt{2} (g+1)\sqrt{g} \leq 2\sqrt{2}(g+1)^{3/2}<35(g+1)^{3/2},$$
proving the claim.

\textbf{Case 2:} $\bar g=1$. In this case $\cX\rightarrow\cX/Q$ ramifies, otherwise $\cX$ itself would be elliptic. Let $\Delta_1$ be a short orbit of $Q$ and let $\Gamma$ be the short orbit of $G$ containing $\Delta_1$. Thus $\Gamma= \Delta_1 \cup \ldots \cup \Delta_k$ where $\Delta_i$ is a short orbit of $Q$ with $|\Delta_i|=|\Delta_1|$ for every $i=2,\dots,k$. For a point $P \in \Delta_1$, $|G_P|=|G|/k|\Delta_1|$ and the quotient group $G_P Q/Q$ fixes the point $\bar P$ of $\cX/Q$ which lies under the orbit $\Delta_1$. From \cite[Theorem 11.94]{hirschfeld-korchmaros-torres2008}
$$|G_P Q / Q|=|G_P|/|Q_P| \leq 24.$$
Moreover, from the Hurwitz genus formula applied to $\cX \rightarrow \cX/Q$,
$$2g-2 \geq 2k|\Delta_1|(|Q_P|-1) \geq \frac{2k|\Delta_1| |Q_P|}{2} \geq \frac{k|\Delta_1| |G_P|}{24}=\frac{|G|}{24}.$$
Since this implies that $|G| \leq 84(g-1)$, this case cannot occur.

\textbf{Case 3:} $\bar g \geq 2$. From the Hurwitz genus formula, we get $|G/Q| >84(\bar g-1)$ as $|G|>84(g-1)$. 

If $Q \not\in \rm{Syl}_2(G)$ then $\bar g$ is even from Proposition \ref{prop20nov2016}. Moreover, the case $Q \in \rm{Syl}_2(G)$ cannot occur since $G/Q$ would be tame and hence $|G/Q| \leq 84(\bar g-1)$, a contradiction.

Thus, we conclude that $\bar g \geq 2$ is even and $\bar g <g$. The curve $\cX/Q$ is ordinary from \cite[Result 2.5]{montanucci- korchmaros}. Hence, the induction hypothesis holds for $\cX/Q$ implying $|G/Q| \leq 35(\bar g+1)^{3/2}$. 
If $|G| > 35(g+1)^{3/5}$, then the Hurwitz genus formula applied to $\cX \rightarrow \cX/Q$ implies that $2g-2 \geq |Q|(2\bar g-2)$. Thus, $|\bar G| \geq 35(\bar g + 1)^{3/2}$, a contradiction. Now the claim follows.
\end{proof}

Combining Theorem \ref{prop1} and Theorem \ref{prop2}, we get the main result of this paper.
\begin{corollary}\label{mariaportamivia}
Let $\cX$ be an ordinary curve of even genus $g \geq 2$ defined over an algebraically closed field $K$ of characteristic $2$. Let $G \leq \aut(\cX)$ be a solvable automorphism group of $\cX$; then $|G| < 35(g+1)^{3/2}$. 
\end{corollary}

\subsection*{Acknowledgments}
The authors wish to thank the anonymous referee for his/her thoughtful reading and useful comments, which greatly improved the exposition and clarity of our paper.

 This research was performed within the activities of  GNSAGA - Gruppo Nazionale per le Strutture Algebriche, Geometriche e le loro Applicazioni of Italian INdAM.
 
 The second author was supported by FAPESP-Brazil, grant 2017/18776-6.


\begin{thebibliography}{999}
\bibitem{bw}N. L. Biggs, A. T. White, 
\emph{Permutation groups and combinatorial structures.}
London Mathematical Society Lecture Note Series, 33. Cambridge University Press, Cambridge-New York, (1979), 140 pp. 
\bibitem{giulietti-korchmaros-2017}
M. Giulietti, G. Korchm\'aros, Algebraic curves with many automorphisms, arXiv:1702.08812.
 \bibitem{Marsh} M. Hall, {\it The Theory of Groups}, Macmillan, New York, (1959), xx+434 pp. 
 \bibitem{henn} H.W. Henn, Funktionenk\"orper mit gro\ss er Automorphismengruppe, \emph{J. Reine  Angew. Math.} {\bf 302} (1978), 96-115.
\bibitem{hirschfeld-korchmaros-torres2008}
J.W.P. Hirschfeld, G. Korchm\'aros and F. Torres,\emph{  Algebraic Curves Over a Finite Field}, Princeton Univ. Press, Princeton, (2008), xiii + 720 pp.
\bibitem{Huppert} B.~Huppert, Zweifach transitive, aufl\"osbare Permutationsgruppen, {\it Math.}, \textbf{68} (1957), 126-150.
\bibitem{lenstra}
H.W. Lenstra, \emph{Galois theory for schemes}, Electronic Third Edition, (2008) vi+109 pp.
\bibitem{montanucci- korchmaros}
G. Korchm\'aros, M. Montanucci, Ordinary algebraic curves with many automorphisms in positive characteristic, arXiv:1610.05252. 
 \bibitem{montanucci-korchmaros-speziali} G. Korchm\'aros, M. Montanucci and P. Speziali, Transcendence Degree One Function Fields Over a Finite Field with Many Automorphisms,  \emph{J. Pure Appl. Algebra} { \bf 222} (2018), no. 7, 1810-1826.
\bibitem{KR} A.~Kontogeorgis, V.~Rotger, On abelian automorphism groups of Mumford curves, \emph{Bull. London Math. Soc.} \textbf{40} (2008), 353-362.
 \bibitem{machi} A.~Mach\`i, \emph{Groups, An introduction to ideas and methods of the theory of groups}, Unitext, {\bf{58}}, Springer, Milan, (2012), xiv+371 pp.


\bibitem{nakajima 1987}
S. Nakajima, $p$-ranks and automorphism groups of algebraic curves, \emph{Trans. Amer. Math. Soc.} {\bf 303} (1987), 595-607. 
\bibitem{serre}
J.-P. Serre, \emph{Local Fields}, Graduate Texts in Mathematics {\bf 67}, Springer, New York, 1979. viii+241 pp.
\bibitem{stbook}
H.~Stichtenoth,  \emph{Algebraic function fields and codes}, Springer-Verlag, Berlin and Heidelberg, (1993), vii+260 pp.
\bibitem{subrao}
D. Subrao, The $p$-rank of Artin-Schreier curves, \emph{Manuscrpta Math.} {\bf 16} (1975), 169-193.
 \bibitem{maddenevalentini1982} R.C.~Valentini, M.L.~Madan,  A
Hauptsatz of L.E. Dickson and Artin--Schreier extensions, \emph{J.
Reine Angew. Math.} {\bf 318} (1980), 156--177.
\end{thebibliography}
\end{document}